\documentclass[10pt]{article}
\usepackage[utf8]{inputenc}
\usepackage{amsmath}
\usepackage{amssymb}
\usepackage{mathrsfs}
\usepackage{amsthm}
\usepackage{mathabx}
\usepackage{tikz-cd}
\usepackage{mathtools}
\usepackage{dsfont}
\usepackage[all]{xy}
\usepackage{stmaryrd}
\usepackage{wasysym}
\usepackage{xparse}

\NewDocumentCommand{\tens}{e{_^}}{%
  \mathbin{\mathop{\otimes}\displaylimits
    \IfValueT{#1}{_{#1}}
    \IfValueT{#2}{^{#2}}
  }%
}
\newcommand{\qhyp}[5]{\,\mbox{}_{#1}\phi_{#2}\!\left(
  \genfrac{}{}{0pt}{}{#3}{#4};#5\right)}

\newcommand{\id}{\mathrm{id}}
\newcommand\restr[2]{{
  \left.\kern-\nulldelimiterspace 
  #1 
  \littletaller 
  \right|_{#2} 
  }}

\newcommand{\littletaller}{\mathchoice{\vphantom{\big|}}{}{}{}}
\theoremstyle{plain}
\newtheorem{thm}{Theorem}[section]
\newtheorem{lem}[thm]{Lemma}
\newtheorem{prop}[thm]{Proposition}
\newtheorem{cor}[thm]{Corollary}

\theoremstyle{definition}

\newtheorem{defn}[thm]{Definition}

\theoremstyle{remark}
\newtheorem{Rem}[thm]{Remark}

\theoremstyle{plain}
\newtheorem*{thm*}{Theorem}
\newtheorem*{lem*}{Lemma}
\newtheorem*{prop*}{Proposition}
\newtheorem*{cor*}{Corollary}
\newtheorem*{conj*}{Conjecture}

\theoremstyle{definition}
\newtheorem*{ass*}{Assumption}
\newtheorem*{defn*}{Definition}

\theoremstyle{remark}
\newtheorem*{Rem*}{Remark}
\newcommand{\compconj}[1]{%
  \overline{#1}%
}
\DeclareMathOperator{\End}{End}
\DeclareMathOperator{\Rep}{Rep}
\DeclareMathOperator{\Hom}{Hom}
\DeclareMathOperator{\spn}{span}

\title{A proper generating functional on a Podle\'{s} sphere}
\author{Masato Tanaka
\footnote{%
tanakamasato.2121@gmail.com, masato.tanaka.c7@math.nagoya-u.ac.jp%
}}
\date{April 2023}

\begin{document}

\maketitle
\begin{abstract}
We construct a proper generating functional $L$ on a Podle\'{s} sphere and we show that $1$-cocycle arising from $L$ coincides with the one in our previous work. We also show that our 1-cocycle is purely non Gaussian and that the full `group' $C^{\ast}$-algebra of the quantum $SL(2,\mathbb{R})$ is liminal.
\end{abstract}
\section{Introduction}
In 2021, De Commer--Dzokou Talla \cite{DCDz1} constructed a quantization $U_q(\mathfrak{sl}(2,\mathbb{R})_t)$, where $0<q<1$, of the universal enveloping algebra of the Lie algebra of the spcial linear group $SL(2,\mathbb{R})$. This quantization is different from the one in the case where $q$ is a root of unity in \cite{M+}. As the no--go theorem by Woronowicz shows, it was hard to find a way to quantize arbitrary locally compact real Lie groups at the operator algebraic level except the compact case. The method De Commer--Dzokou Talla found is, however, applicable to any locally compact real Lie groups and is compatible with the framework of operator algebras.\par 
Since the quantum $SL(2,\mathbb{R})$ is a newly-born object and a quantization as as a coideal (not a Hopf algebra), there are of course many things to consider. For example, De Commer--Dzokou Talla found the appropriate classes of representations of $U_q(\mathfrak{sl}(2,\mathbb{R})_t)$ and classified irreducible ones \cite{DCDz1}. They also defined the monoidal structure of the representation category of the `convolution algebra' of the quantum $SL(2,\mathbb{R})$ and constructed an invariant integral on the `convolution algebra' of the quantum $SL(2,\mathbb{R})$ \cite{DCDz2}.
Based on \cite{DCDz1} and \cite{DCDz2}, the author of the present paper showed an irreducible $\ast$-representation induces an $SL(2,\mathbb{R})_t$-admissible $\ast$-representation and constructed a proper $1$-cocycle on arbitrary (non-standard) Podle\'{s} sphere and extended to a $1$-cocycle on the `convolution algebra' of the quantum $SL(2,\mathbb{R})$ \cite{MT}. The representation space of the cocycle in \cite{MT} is the direct sum of discrete series representations in \cite{DCDz1}, which meas that this cocycle is $L^2$-cohomology and we can expect some applications. Since the cocycle constructed in \cite{MT} and the generating functional which will be constructed in the present paper are proper, we should consider that the quantum $SL(2,\mathbb{R})$ is not of property (T) and is of Haagerup property. Thus, in later works, we should give appropriate definitions and characterizations of approximation properties such as property (T) and Haagerup property of Drinfeld double coideals. Considering the celebrated Schoenberg correspondence, we should construct a (proper) generating functional, a convolution semigroup of positive linear functionals, a L\'{e}vy process et cetera associated to the cocycle in \cite{MT}.\par
In this paper we consider generating functional and we give some observations of properties the quantum $SL(2,\mathbb{R})$ has. Our main results are the following:
\begin{itemize}
    \item We have a proper generating functional on $\mathcal{O}_q(S_t^2)$ whose associated 1-cocycle coincides with the proper 1-cocycle $C\colon\mathcal{O}_q(S_t^2)\to\mathcal{D}_2^+\oplus\mathcal{D}_2^-$ in \cite{MT}.  
    \item Our $1$-cocycle $C$ is purely non-Gaussian.
    \item The quantum $SL(2,\mathbb{R})$ is `liminal' in the sense that the $C^*$-completion of $C_c(SL_q(2,\mathbb{R}))$ with respect to the largest $C^*$-norm is liminal.
\end{itemize}
We will also give some proofs which were omitted in \cite{MT}.\\
Acknowledgements: This work was supported by the “Nagoya University 
Interdisciplinary Frontier Fellowship” by Nagoya University and JST, the establishment of university fellowships towards the creation of science technology innovation, Grant Number JPMJFS2120. This work was supported by the Research Institute for Mathematical Sciences, an International Joint Usage/Research Center located in Kyoto University.
\section{Preliminaries}
In this section, we recall the definitions of well-known objects which we need in the present paper. The best general references here are \cite{DCDz1}, \cite{DCDz2} and \cite{KS}.
\subsection{Notations and conventions}
\begin{itemize}
    \item Our vector spaces are over the complex field $\mathbb{C}$ unless otherwise stated.
    \item We denote the set of bounded operators on a normed space $X$ by $\mathbb{B}(X)$. 
    \item We denote the set of compact operators on a normed space $X$ by $\mathbb{K}(X)$. 
    \item For a preHilert space $\mathcal{H}$, let $\End({\mathcal{H}})$ denote the space of all adjointable linear maps on $\mathcal{H}$.
    \item Our inner products are linear in the second variable and conjugate linear in the first variable.
    \item We put $[a]\coloneqq \dfrac{q^a-q^{-a}}{q-q^{-1}}$, $\llbracket a\rrbracket\coloneqq q^a-q^{-a}$ and  $\{a\}\coloneqq q^a+q^{-a}$, where $a\in\mathbb{R}$.
    \item We use the $q$-Pochhammer symbols:
\begin{align*}
 (b;q)_\infty &\coloneqq\prod_{i=0}^\infty (1-b q^i),\\
 (b;q)_n&\coloneqq(b;q)_\infty / (b q^n;q)_\infty ,\\
(b_1,b_2,\ldots,b_r;q)_n&\coloneqq(b_1;q)_n(b_2;q)_n\cdots(b_r;q)_n.
\end{align*}
\item We use $q$-hypergeometric series:
\begin{align*}
   \qhyp {s+1}s{a_1,\ldots,a_{s+1}}{b_1,\ldots, b_s}{q,z}\coloneqq\displaystyle\sum_{i=0}^{\infty}\dfrac{(a_1,a_2,\ldots,a_{s+1};q)_iz^i}{(b_1,b_2,\ldots,b_s;q)_i(q;q)_i}_{.}
\end{align*}
\end{itemize}
\begin{defn}
We define $U_q(\mathfrak{su}(2))$ to be the universal unital algebra generated by the elements $k, k^{-1}, e$ and $f$ with the relations
\begin{align*}
    kk^{-1}=k^{-1}k=1,\quad ke=q^2ek,\quad kf=q^{-2}fk,\quad [e,f]=\dfrac{k-k^{-1}}{q-q^{-1}}.
\end{align*}
The Hopf $\ast$-algebra structure is given by the following:
\begin{align*}
 &\Delta(k^{\pm})=k^{\pm}\otimes k^{\pm},\quad\Delta(e)=e\otimes 1+k\otimes e,\quad\Delta(f)=1\otimes f+f\otimes k^{-1}\\
&\qquad\qquad\qquad\qquad\epsilon(k^{\pm})=1,\quad\epsilon(e)=0,\quad\epsilon(f)=0\\
    &\qquad\quad\quad\quad S(k^{\pm})=k^{\mp},\quad S(e)=-k^{-1}e,\quad S(f)=-fk\\
    &\qquad\qquad\qquad\qquad
    k^*=k,\quad e^*=fk,\quad f^*=k^{-1}e.
\end{align*}\end{defn}
The Hopf $\ast$-algebra $U_q(\mathfrak{su}(2))$ has a $(2s+1)$-dimensional $\ast$-representation for each $s\in\tfrac{1}{2}\mathbb{Z}_{\geq0}$. We denote them by $(\pi_s,H_s)$. We fix an orthonormal basis $\{\xi_i^s\}_{i=-s}^s$ such that
 \begin{align*}
     \pi_s(k)\xi_i^s&=q^{2i}\xi_i^s,\\
    \pi_s(e)\xi_i^s&=q^{i+1}([s-i][s+i+1])^{1/2}\xi_{i+1}^s,\\
    \pi_s(f)\xi_i^s&=q^{-i}([s+i][s-i+1])^{1/2}\xi^s_{i-1}.
 \end{align*}
\begin{defn}
    The algebra
$\mathcal{O}_q(SU(2))$ is the unital associative algebra generated by the elements $\alpha, \gamma$ with the relations
\begin{align*}
\alpha^{\ast}\alpha+\gamma^{\ast}\gamma=1, \alpha\alpha^{\ast}+q^2\gamma\gamma^{\ast}=1, \gamma^{\ast}\gamma=\gamma\gamma^{\ast}, \alpha\gamma=q\gamma\alpha, \alpha\gamma^{\ast}=q\gamma^{\ast}\alpha.
\end{align*}
The Hopf $\ast$-algebra structure is given by the following:
\begin{align*}
&\quad\qquad\qquad\Delta(\alpha)=\alpha\otimes\alpha-q\gamma^{\ast}\otimes\gamma, \Delta(\gamma)=\gamma\otimes\alpha+\alpha^{\ast}\otimes\gamma\\
    &\epsilon(\alpha)=1, \epsilon(\gamma)=0, S(\alpha)=\alpha^{\ast}, S(\alpha^{\ast})=\alpha, S(\gamma)=-q\gamma, S(\gamma^{\ast})=-q^{-1}\gamma^{\ast}.
\end{align*}
\end{defn}
These two Hopf $\ast$-algebras are in duality:
\begin{defn}
    The unitary pairing between $U_q(\mathfrak{su}(2))$ and $\mathcal{O}_q(SU(2))$ is given by
    \begin{align*}
        (\begin{bmatrix}
\alpha & -q\gamma^{\ast} \\
\gamma & \alpha^{\ast} \\
\end{bmatrix}, k)&=\begin{bmatrix}
q & 0 \\
0 & q^{-1} \\
\end{bmatrix},
(\begin{bmatrix}
\alpha & -q\gamma^{\ast} \\
\gamma & \alpha^{\ast} \\
\end{bmatrix}, e)=\begin{bmatrix}
0 & q^{1/2} \\
0 & 0 \\
\end{bmatrix},\\
(\begin{bmatrix}
\alpha & -q\gamma^{\ast} \\
\gamma & \alpha^{\ast} \\
\end{bmatrix}, f)&=\begin{bmatrix}
0 & 0 \\
q^{-1/2} & 0 \\
\end{bmatrix}_{.}
    \end{align*}
\end{defn}
By this pairing we have actions:
\begin{defn}
   We define the actions as follows.
   \begin{align*}
   h\rhd a&\coloneqq a_{(1)}(a_{(2)}, h)=(\id\otimes (\bullet, h))\Delta(a)\\
a\lhd h&\coloneqq (a_{(1)}, h)a_{(2)}=((\bullet, h)\otimes\id)\Delta(a)\\
a\rhd h&\coloneqq h_{(1)}(a, h_{(2)})=(\id\otimes (a,\bullet))\Delta(h)\\
h\lhd a&\coloneqq (a, h_{(1)})h_{(2)}=((a, \bullet)\otimes\id)\Delta(h),
   \end{align*}
   where $a\in\mathcal{O}_q(SU(2))$ and $h\in U_q(\mathfrak{su}(2))$.
\end{defn}
As in De Commer--Dzokou Talla's paper \cite{DCDz1}, for $t=\llbracket a \rrbracket$ we define the elements $\Tilde{B}_t$ and $B_t$ in the Hopf $\ast$-algebra $U_q(\mathfrak{su}(2))$ by \[\Tilde{B}_t\coloneqq q^{-1/2}(e-fk)-\sqrt{-1}(q-q^{-1})^{-1}tk+\sqrt{-1}(q-q^{-1})^{-1}t,\]
\[B_t\coloneqq q^{-1/2}(e-fk)-\sqrt{-1}(q-q^{-1})^{-1}t.\]
In the convention of \cite{DCDz1} and \cite{DCDz2}, we define $I_l\coloneqq \mathbb{C}[B_t]$ and $\mathcal{O}_q(S_t^2)_r\coloneqq I_l^{\perp}\coloneqq\{a\in\mathcal{O}_q(SU(2))\mid a\lhd B_t=\epsilon(B_t)a\}$. In this case $I_l$ is a left coideal of $U_q(\mathfrak{su}(2))$ and the Podle\'{s} sphere $\mathcal{O}_q(S_t^2)_r$ is a right coideal of $\mathcal{O}_q(SU(2))$. As in \cite{DCDz1} we define $U_q(\mathfrak{sl}(2,\mathbb{R})_t)_{r,l}\coloneqq I_l^{\perp}\bowtie I_l$. In the convention of \cite{MT} we define $I_r\coloneqq\mathbb{C}[R(B_t)]$ and $\mathcal{O}_q(S_t^2)_l\coloneqq I_r^{\perp}\coloneqq\{a\in\mathcal{O}_q(SU(2))\mid R(B_t)\rhd a=\epsilon(B_t)a\}$, where $R$ is a unitary antipode of $U_q(\mathfrak{su}(2))$. In this case, the algebra $I_r$ is a right coideal of $U_q(\mathfrak{su}(2))$, the Podle\'{s} sphere $\mathcal{O}_q(S_t^2)_l$ is a left coideal of $\mathcal{O}_q(SU(2))$ and we define $U_q(\mathfrak{sl}(2,\mathbb{R})_t)_{l,r}\coloneqq I_r^{\perp}\bowtie I_r$.\par
Note that $\mathcal{O}_q(SU(2))\simeq\bigoplus_{s\in\tfrac{1}{2}\mathbb{Z}_{\geq0}}\mathbb{B}(H_s)^{\ast}$. Put $\mathcal{U}\coloneqq\Hom(\mathcal{O}_q(SU(2)),\mathbb{C})\simeq\prod_{s\in\tfrac{1}{2}\mathbb{Z}_{\geq0}}\mathbb{B}(H_s)$ and $\mathcal{U}_{\infty}\coloneqq l_{\infty}-\prod_{s\in\tfrac{1}{2}\mathbb{Z}_{\geq0}}\mathbb{B}(H_s)$. Let $\mathcal{A}=\mathcal{O}_q(SU(2))$ and $\mathcal{B}_+=\mathcal{B}\cap\ker(\epsilon)$, where $\mathcal{B}=\mathcal{O}_q(S_t^2)_i$ and $i=l, r$. If we take the convention in \cite{DCDz1} and \cite{DCDz2}, we let $\mathcal{I}_l\coloneqq\Hom(\mathcal{A}/\mathcal{A}\mathcal{B}_+,\mathbb{C})\simeq\prod_{n\in\mathbb{Z}}\mathbb{B}(\mathbb{C}_{n,l})$, where $\mathbb{C}_{n,l}$ is a one dimensional vector space on which $\sqrt{-1}B_t$ acts as $[a+n]$.  If we take the convention in \cite{MT}, we put $\mathcal{I}_r\coloneqq\Hom(\mathcal{A}/\mathcal{B}_+\mathcal{A},\mathbb{C})\simeq\prod_{n\in\mathbb{Z}}\mathbb{B}(\mathbb{C}_{n,r})$, where $\mathbb{C}_{n,r}$ is a one dimensional vector space on which $\sqrt{-1}R(B_t)$ acts as $[a+n]$. We put $\mathcal{I}_{c,i}\coloneqq\bigoplus_n\mathbb{B}(\mathbb{C}_{n,i})$ and $\mathcal{I}_{\infty,i}\coloneqq l_{\infty}-\prod_n\mathbb{B}(\mathbb{C}_{n,i})$ where $i=l,r$. We define $C_c(SL_q(2,\mathbb{R})_t)_{i,j}\coloneqq\mathcal{O}_q(S_t^2)_i\bowtie\mathcal{I}_{c,j}$ where $(i,j)=(l,r),(r,l)$.\par
For simplicity we shall omit the subscripts $l$ and $r$ in the sequel.
\section{Observations of conventions}
\subsection{The algebras $\mathcal{U}_q$ and $U_q(\mathfrak{su}(2))$}
In the paper \cite{K} written by Koornwinder the algebra $\mathcal{U}_q$ is defined as the unital associative algebra generated by the elements $1,A,B,C$ and $D$ with the relations
\begin{align*}
AD=DA=1,\quad AB=qBA,\quad AC=q^{-1}CA, \quad [B,C]=\dfrac{A^2-D^2}{q-q^{-1}}_{.}   \end{align*}
This $\mathcal{U}_q$ has a Hopf $\ast$-algebra structure as follows:
\begin{align*}
   &\Delta(A)=A\otimes A,\quad \Delta(D)=D\otimes D ,\\
   &\Delta(B)=A\otimes B+B\otimes D, \quad\Delta(C)=A\otimes C+C\otimes D,\\
   &\epsilon(A)=\epsilon(D)=1,\quad \epsilon(B)=\epsilon(C)=0,\\
   &S(A)=D,\quad S(D)=A,\quad S(B)=-q^{-1}B,\quad S(C)=-qC,\\
   &A^{*}=A,\quad D^{*}=D,\quad B^{*}=C,\quad C^{*}=B.
\end{align*}
The Hopf $\ast$-algebra $\mathcal{U}_q$ has a $(2s+1)$-dimensional $\ast$-representation $(t^{s},H^{s})$ for each $s\in\tfrac{1}{2}\mathbb{Z}_{\geq0}$. We fix an orthonormal basis $\{e^s_i\}_{i=-s,-s+1,\ldots,s}$ of $H^{s}$ such that
\begin{align*}
    &t^{s}(A)e^s_i=q^{-i}e^s_i,\quad t^{s}(D)e^s_i=q^ie^s_i\\
    &t^{s}(B)e^s_i=([s-i+1][s+i])^{1/2}e^s_{i-1},\\
    &t^{s}(C)e^s_i=([s-i][s+i+1])^{1/2}e^s_{i+1}.
\end{align*}
In this Hopf $\ast$-algebra we define the element $X_{-a}$ by \[X_{-a}\coloneqq\sqrt{-1}q^{1/2}B-\sqrt{-1}q^{-1/2}C+(q-q^{-1})^{-1}t(A-D),\] where $t\coloneqq\llbracket a\rrbracket$. The kernel of the operator $t^{s}(X_{-a})$ is $1$-dimensional and spanned by $\sum_{i=-s}^sq^{-\tfrac{1}{2}i}c^{s,-a}_ie^s_i$ if $s$ is a non-negative integer \cite[Lemma 4.6]{K}. Similarly the kernel of $t^{s}(X^{\ast}_{-a})$ is $1$-dimensional and spanned by $\sum_{i=-s}^sq^{\tfrac{1}{2}i}c^{s,-a}_ie^s_i$ when $s$ is a non-negative integer \cite[Lemma 4.6]{K}. Here \[c_i^{s,-a}\coloneqq\dfrac{(\sqrt{-1})^iq^{-(s-a)i}q^{i^2/2}}{(q^2;q^2)_{s+i}^{1/2}(q^2;q^2)_{s-i}^{1/2}}\qhyp 32{q^{-2s+2i},q^{-2s},-q^{-2s+2a}}{q^{-4s},0}{q^2,q^2}=c_{-i}^{s,-a}.\]
By identifying $e^s_{-i}=\xi_i^s$, we have $H^s=H_s$. As operators on the Hilbert space $H_s$, we have \[\sqrt{-1}R(\Tilde{B}_t)=k^{-1/2}X_{-a},\] where $R$ is the unitary antipode of $U_q(\mathfrak{su}(2))$. 
\begin{Rem} 
The two algebras are related in such a way that 
\begin{align*}
    A=k^{1/2},\quad B=k^{-1/2}e, \quad C=fk^{1/2},\quad D=k^{-1/2}.
\end{align*}
Note that this is a formal expression.
\end{Rem}
Let us give a proof of \cite[Theorem 2.15]{MT}. The computations are essentially the same as in \cite{K}. Let $u_{i,j}^s\coloneqq u_{\xi_i^s,\xi_j^s}^s$ for $s\in\tfrac{1}{2}\mathbb{Z}_{\geq0}$. Let $n\in\mathbb{Z}_{\geq0}$. Let $\eta_{[a+2i]}^{n} \in H_{n}$ be a unit vector which is an eigenvector for $\pi_{n}(\sqrt{-1}R(B_t))$ with the eigenvalue $[a+2i]$. Then $\{u^{n}_{\eta_{[a+2i]}^{n}, \eta_{[a]}^{n}} \mid n\in \mathbb{Z}\geq0, i=-n, \ldots n\}$ forms a basis for the left coideal $\mathcal{O}_{q}(S_{t}^{2})$. Let $u_{[a+2i],[a+2j]}^n\coloneqq  u_{\eta_{[a+2i]}^{n}, \eta_{[a+2j]}^{n}}^n$. We let $\eta_{[a]}^{n}\coloneqq\lVert \xi^{(\llbracket a\rrbracket,1)} \rVert^{-1}\xi^{(\llbracket a\rrbracket,1)}$, where $\xi^{(\llbracket a\rrbracket,1)}=q^{1/2}\xi_{1}^{1}-(q+q^{-1})^{-1/2}\sqrt{-1}\llbracket a\rrbracket\xi_{0}^{1}+q^{-1/2}\xi_{-1}^1$ is a spherical vector. Write $\eta_{[a]}^{n}=\displaystyle\sum_{j=-n}^{n}a_{j}^n\xi^n_j$, where $a_i\in\mathbb{C}$. Then $u_{[a],[a]}^n=\displaystyle\sum_{i,j=-n}^{n} \compconj{a_{i}^n}{a_{j}^n}u_{i,j}^n$.
We have $\mathcal{O}_q(SU(2))
=\bigoplus_{s\in\tfrac{1}{2}\mathbb{Z}_{\geq0}}\mathcal{O}_q(SU(2))_s$, where $\mathcal{O}_q(SU(2))_s\coloneqq \spn\{u_{i,j}^s\mid i,j=-s,\ldots s\}$. By the Clebsch–-Gordan rule, any polynomials in $u_{[a],[a]}^1$ of degree less than or equal to $n$ belong to the direct sum of $\mathcal{O}_q(SU(2))_k$'s ($k=0,\ldots,n$) and they are $\Phi_{\mathcal{C}}$-bi-invariant.
For $\theta\in\mathbb{R}$, define an algebra homomorphism $\pi_{\theta/2}^{1}\colon\mathcal{O}_q(SU(2))\to\mathbb{C}$ by 
\[
\pi_{\theta/2}^{1}(\begin{bmatrix}
\alpha & -q\gamma^{\ast} \\
\gamma & \alpha^{\ast} \\
\end{bmatrix})
=\begin{bmatrix}
e^{\sqrt{-1}\theta/2} & 0 \\
0 & e^{-\sqrt{-1}\theta/2} \\
\end{bmatrix}_{.}
\]
For this map, we have $\pi_{\theta/2}^{1}(u_{i,j}^s)=\delta_{i,j}e^{\sqrt{-1}j\theta}$ for all $s,i$ and $j$, where $\delta_{i,j}$ is the Kronecker delta. From this we deduce that $\pi_{\theta/2}^{1}(u_{[a],[a]}^1)=\lVert \xi^{(\llbracket a\rrbracket,1)} \rVert^{-2}(qe^{\sqrt{-1}\theta}+q^{-1}e^{-\sqrt{-1}\theta}+(q+q^{-1})^{-1}\llbracket a\rrbracket^2)$. Hence the set $\{(u^{1}_{[a],[a]})^k\}_{k\geq0}$ is linearly independent. Then we have
\begin{lem}(cf.\ \cite[Proposition 4.7.]{K})\label{AW}
The element $u_{[a],[a]}^n$ is a polynomial in $u_{[a],[a]}^1$ of degree $n$.
\end{lem}
Thus we see that there exists a sequence $\{P_n\}_{n\geq0}$ of polynomials such that $P_0=1$, $\deg{P_n}=n$ and $u_{[a],[a]}^n=P_{n}(u_{[a],[a]}^1)$ for all $n\geq1$. Apply the map $\pi_{\theta/2}^{1}$ to the identity $\displaystyle\sum_{i,j=-n}^{n}\compconj{a_{i}^n}{a_{j}^n}u_{i,j}^n=P_{n}(u_{[a],[a]}^1)$, then we have
\[\displaystyle\sum_{j=-n}^{n}|a^n_{j}|^{2}e^{\sqrt{-1}j\theta}=P_{n}(\lVert \xi^{(\llbracket a\rrbracket,1)} \rVert^{-2}(qe^{\sqrt{-1}\theta}+q^{-1}e^{-\sqrt{-1}\theta}+(q+q^{-1})^{-1}\llbracket a\rrbracket^2)).\]
Put $z=qe^{\sqrt{-1}\theta}$. We have
\[\displaystyle\sum_{j=-n}^{n}|a^n_{j}q^{-j/2}|^{2}z^j=P_{n}(\dfrac{z+z^{-1}+(q+q^{-1})^{-1}\llbracket a\rrbracket^2}{q+q^{-1}+(q+q^{-1})^{-1}\llbracket a\rrbracket^2}).\]
Now let $Q_n(X)\coloneqq P_n(\tfrac{2X+(q+q^{-1})^{-1}\llbracket a\rrbracket^2}{q+q^{-1}+(q+q^{-1})^{-1}\llbracket a\rrbracket^2})$. Then we have
\begin{align*}
  \displaystyle\sum_{j=-n}^{n}|a^n_{j}q^{-j/2}|^{2}z^j&=Q_n((z+z^{-1})/2)  
\end{align*}
for all $z\in\{z\in\mathbb{C}\mid |z|=q\}$. By the identity theorem, this equation holds for all $z\in\{z\in\mathbb{C}\mid 1/2<|z|<2\}$, in particular for all $z$ with $|z|=1$. Therefore by letting $z=e^{\sqrt{-1}\theta}$, we have
\begin{align*}
    \displaystyle\sum_{j=-n}^{n}|a_{j}^{n}q^{-j/2}|^2e^{\sqrt{-1}j\theta}&=Q_n(\cos\theta).
\end{align*}
Apply $k^{-1/2}$ to the identity
\[\eta^s_{[a]}=\displaystyle\sum_{i=-n}^na^n_i\xi^n_i\]
then we get
\[k^{-1/2}\eta^s_{[a]}=\displaystyle\sum_{i=-s}^sq^{-i}a^s_i\xi^s_i\in\ker(X^{\ast}_{-a})=\mathbb{C}\displaystyle\sum_{i=-s}^sq^{-\tfrac{1}{2}i}c^{s,-a}_{-i}\xi^s_i.\]
Hence there exists some $\alpha_n\in\mathbb{C}$ such that $\alpha_na_i^n=q^{\tfrac{1}{2}i}c^{n,-a}_{-i}$ for all $i$. We can compute the absolute value of $\alpha_n$ by $\sum_i|a_i^n|^2=1$ and then
\[|\alpha_n|^2=\displaystyle\sum_{i=-n}^nq^i|c_{-i}^{n,-a}|^2=\displaystyle\sum_{i=-n}^nq^i|c_{i}^{n,-a}|^2.\]
It follows that
\begin{align*}
    \displaystyle\sum_{i=-n}^{n}|c_{i}^{n,-a}|^2e^{\sqrt{-1}i\theta}&=|\alpha_n|^2Q_n(\cos\theta).
\end{align*}
 Consulting \cite[Lemma 4.6., Remark 4.8.]{K}, we get
\begin{align*}
    Q_n(\bullet)=|\alpha_n|^{-2}|c^{n,-a}_n|^2(q^{2n+2};q^2)_{n}^{-1}p_n(\bullet),
\end{align*}
where the Askey--Wilson polynomial $p_n$ is defined by
\begin{align*}
   & p_n(\cos\theta; -q^{-2a+1}, -q^{2a+1}, q, q| q^2)\\
    &=(-1)^{n}q^{n(-2a+1)}(q^2,-q^{-2a+2},-q^{-2a+2};q)_n\\
    &\cdot\qhyp 43{q^{-n},q^4q^{n-1},-q^{-2a+1}e^{\sqrt{-1}\theta},-q^{-2a+1}e^{-\sqrt{-1}\theta}}{q^2,-q^{-2a+2},-q^{-2a+2}}{q,q},
\end{align*}
and
\begin{align*}
    c^{n,-a}_n&=\dfrac{(\sqrt{-1})^{n}q^{-(n-a)n}q^{n^2/2}}{(q^2;q^2)^{1/2}_{2n}}
    \qhyp 32{1,q^{-2n},-q^{-2n+2a}}{q^{-4n},0}{q^2,q^2}\\
    &=\dfrac{(\sqrt{-1})^{n}q^{-(n-a)n}q^{n^2/2}}{(q^2;q^2)^{1/2}_{2n}}_{.}
\end{align*}
\subsection{Left v.s. right}
 In \cite{DCDz1} and \cite{DCDz2}, the algebras $U_q(\mathfrak{sl}(2,\mathbb{R})_t)$ and $\mathcal{O}_q(S_t^2)\bowtie\mathcal{I}_c$ are unital right coideal. On the other hand, in \cite{MT}, they are unital left coideal. The author of \cite{MT} takes his convention in order to consider the growth of $1$-cocycle on $SL_q(2,\mathbb{R})$. Let $\mathcal{D}_2^{\pm}$ be the discrete series representations of $U_q(\mathfrak{sl}(2, \mathbb{R})_t)$ \cite[Theorem 3.17, Remark3.18]{DCDz1}. If we take the convention in \cite{DCDz1} and \cite{DCDz2}, construct a $1$-cocycle $C\colon\mathcal{O}_q(S_t^2)\to\mathcal{D}^+_2\oplus\mathcal{D}^-_{2}$ in the same way as in \cite{MT} and consider the growth of our $1$-cocycle $C^{n\ast}C^{n}$ as in \cite[Definition 3.2]{MT}, then $C^{n\ast}C^{n}$ is a $(2n+1)\times(2n+1)$-matrix:
\begin{align*}
    C^{n\ast}C^{n}&=\sum_{i,j=-n}^{n}(\eta^n_{[a]}\eta^{n\ast}_{[a+2i]})^{\ast}\eta^n_{[a]}\eta^{n\ast}_{[a+2j]}C(u^n_{[a],[a+2i]})^{\ast}C(u^n_{[a],[a+2j]})\\
    &=\sum_{i=-n}^{n}\eta^{n}_{[a+2i]}\eta^{n\ast}_{[a+2i]}C(u^n_{[a],[a+2i]})^{\ast}C(u^n_{[a],[a+2i]})\\
    &=\begin{bmatrix}
a_{n,n} & \cdots & 0 & \cdots & 0\\
\vdots & \ddots &        &        & \vdots \\
0 &        & 0 &        & 0 \\
\vdots &        &        & \ddots & \vdots \\
0 & \cdots & 0 & \cdots & a_{-n,-n}
\end{bmatrix}_{.}
\end{align*}
where $a_{i,i}\coloneqq C(u^n_{[a],[a+2i]})^{\ast}C(u^n_{[a],[a+2i]})$. Since the $(n,n)$-th position of this matrix is zero, the inequality $C^{n\ast}C^{n}\geq M I_{2n+1}$ implies $M=0$. Hence we cannot deduce that $C$ is proper in the sense of \cite[Definition 3.2]{MT}. Depending on purposes, we will switch conventions.
\section{Observations of representations}
In this section we take the convention in \cite{DCDz1} and \cite{DCDz2}. We consider representations of the quantum $SL(2,\mathbb{R})$. In \cite{DCDz1} the module $\mathcal{M}_{\lambda, b}$ was introduced. De Commer--Dzokou Talla constructed a $\ast$-invariant sesquilinear form on $\mathcal{M}_{\lambda, b}$. They let $\mathcal{N}_{\lambda, b}$ be its kernel and put $\mathcal{L}_{\lambda, b}\coloneqq\mathcal{M}_{\lambda, b}/\mathcal{N}_{\lambda, b}$ to construct and classify the irreducible admissible representations. First we give an easy remark on the module $\mathcal{L}_{\lambda, b}$. 
\begin{prop}
    The module $\mathcal{L}_{\lambda, b}$ is always irreducible.
\end{prop}
\begin{proof}
   Suppose $\mathcal{K}\subset\mathcal{L}_{\lambda, b}$ is a nonzero invariant subspace. Let $p\colon\mathcal{M}_{\lambda, b}\twoheadrightarrow\mathcal{L}_{\lambda, b}$ be a canonical map. We have $\mathcal{N}_{\lambda, b}\subsetneq p^{-1}(\mathcal{K})\subset\mathcal{M}_{\lambda, b}$. Then we have $e_a\perp p^{-1}(\mathcal{K})$ or $e_a\in p^{-1}(\mathcal{K})$. If the first one holds then we have $0=\langle (T_a^{\pm,(n)})^*\xi,e_a\rangle=\langle\xi,e_{a\pm 2n}\rangle$ for all $\xi\in p^{-1}(\mathcal{K})$ and $n$. This means that $\langle\xi,\eta\rangle=0$ for all $\xi\in p^{-1}(\mathcal{K}),  \eta\in\mathcal{M}_{\lambda, b}$, which means that $p^{-1}(\mathcal{K})=\mathcal{N}_{\lambda, b}$. This is a contradiction. Then we must have $e_a\in p^{-1}(\mathcal{K})$, which implies that 
   $e_{a\pm 2n}\in p^{-1}(K)$. It follows that $p^{-1}(\mathcal{K})=\mathcal{M}_{\lambda, b}$. Hence $\mathcal{K}=\mathcal{L}_{\lambda, b}$.
\end{proof}
Next, let us consider the figures in \cite[p.15, 16]{DCDz1}. In fact these figures respect Fell's topology on the set of (equivalence classes of) $SL(2,\mathbb{R})_t$-admissible irreducible representations. For example, the principal representation $\mathcal{L}_{\lambda}^+$ converges to $\mathds{1}$ and $\mathcal{D}_2^{\pm}$ in Fell's topology. More precisely, we have the following proposition (cf.\ \cite[Proposition 4.10]{G}).
\begin{prop}
  Let $a\in[0,1/2]$ and put $t=\llbracket a\rrbracket$. The topology on the set of (equivalence classes of) $SL(2,\mathbb{R})_t$-admissible irreducible representations is a usual topology in $\mathbb{R}^2$ except the following:
  \begin{itemize}
  \item When $a=0$,
  \begin{enumerate}
      \item $\mathcal{L}_{\lambda}^{+}\to\mathds{1}, \mathcal{D}_{2}^{\pm}$ as $\lambda\to q+q^{-1}$,
      \item $\mathcal{L}_{\lambda}^{+}\to\mathcal{C}, \mathcal{E}_{2}^{\pm}$ as $\lambda\to -q-q^{-1}$,
      \item $\mathcal{L}_{\lambda}^{-}\to\mathcal{D}_{1}^{\pm}$ as $\lambda\to 2$,
      \item $\mathcal{L}_{\lambda}^{-}\to\mathcal{E}_{1}^{\pm}$ as $\lambda\to -2$.
  \end{enumerate}
  \item When $a=1/2$,
  \begin{enumerate}
      \item $\mathcal{L}_{\lambda}^{+}\to\mathds{1}, \mathcal{D}_{2}^{\pm}$ as $\lambda\to q+q^{-1}$,
      \item $\mathcal{L}_{\lambda}^{+}\to\mathcal{E}_{2}^{\pm}$ as $\lambda\to -2$,
      \item $\mathcal{L}_{\lambda}^{-}\to\mathcal{D}_{1}^{\pm}$ as $\lambda\to 2$,
      \item $\mathcal{L}_{\lambda}^{-}\to\mathcal{C}, \mathcal{E}_{1}^{\pm}$ as $\lambda\to -q-q^{-1}$.
  \end{enumerate}
  \item When $0<a<1/2$,
  \begin{enumerate}
    \item $\mathcal{L}_{\lambda}^{+}\to\mathds{1}, \mathcal{D}_{2}^{\pm}$ as $\lambda\to q+q^{-1}$,
      \item $\mathcal{L}_{\lambda}^{+}\to\mathcal{E}_{2}^{\pm}$ as $\lambda\to -q^{2a-1}-q^{-2a+1}$,
      \item $\mathcal{L}_{\lambda}^{-}\to\mathcal{D}_{1}^{\pm}$ as $\lambda\to 2$,
      \item $\mathcal{L}_{\lambda}^{-}\to\mathcal{E}_{1}^{\pm}$ as $\lambda\to -q^{2a}-q^{-2a}$.  
  \end{enumerate}
  \end{itemize}
\end{prop}
Note that the computation in the proof of \ref{gen} proves that the principal series representation converges to the trivial representation. In a similar way, we can prove the proposition above.\par
Let us prove the correspondence theorem stated before Lemma 2.13 of \cite{MT}.
\begin{prop}Let $\mathcal{H}$ be a preHilbert space.
There exists a one to one correspondence between
\begin{enumerate}
    \item Admissible $\ast$-representation of $U_q(\mathfrak{sl}(2,\mathbb{R})_t)$ on $\mathcal{H}$
    \item Unital continuous $\ast$-representation of $\mathcal{O}_q(S_t^2)\bowtie\mathcal{I}$ on $\mathcal{H}$, where $\Phi_n\mathcal{H}$ is finite dimensional for each $n$.
    \item Non-degenerate  $\ast$-representation of $\mathcal{O}_q(S_t^2)\bowtie\mathcal{I}_c$ on $\mathcal{H}$, where $\Phi_n\mathcal{H}$ is finite dimensional for each $n$.
\end{enumerate}
\end{prop}
\begin{proof}
We denote representations in 1, 2, 3 by $\pi_1, \pi_2, \pi_3$ respectively. By restricting, we get $\pi_1$ from $\pi_2$. In a similar way, we get $\pi_3$ from $\pi_2$. By letting $\pi_2(\sum b_n\Phi_n)\Phi_m\xi\coloneqq\pi_3(b_m\Phi_m)\xi$, where $\xi\in\mathcal{H}$ and $m\in\mathbb{Z}$, we get $\pi_2$ from $\pi_3$. By letting $\pi_1(\sqrt{-1}B_t)\Phi_m\xi\coloneqq[a+m]\Phi_m\xi$, where $\xi\in\mathcal{H}$ and $m\in\mathbb{Z}$. 
\end{proof}
In \cite{DCDz2}, the $k^{-1}$-invariant integral $\varphi_t$ on $SL_q(2,\mathbb{R})$ and the regular representation of $SL_q(2,\mathbb{R})$ were constructed. We show that they are faithful on $C_c(SL_q(2,\mathbb{R})_t)$.
\begin{prop}
The integral $\varphi_t$ is faithful on $C_c(SL_q(2,\mathbb{R})_t)$.
\end{prop}
\begin{proof}
Let $x\in\mathcal{O}_q(S_t^2)\bowtie\mathcal{I}_c$ be given and write $x=\sum_{n}\Phi_{n}b_n$, where $b_n\in\mathcal{O}_q(S_t^2)$. If $\varphi_t(x^{\ast}x)=0$, then 
\begin{align*}
    0&=\varphi_t((\displaystyle\sum_{n}\Phi_{n}b_n)^{\ast}\displaystyle\sum_{m}\Phi_{m}b_m)\\
    &=\varphi_t(\displaystyle\sum_{n,m}b^{\ast}_n\Phi_{n}\Phi_{m}b_m)\\
    &=\displaystyle\sum_{n}\varphi_t(b^{\ast}_{n}\Phi_{n}b_n)\\
    &=\displaystyle\sum_{n}\varphi_t(\Phi_{n}b_{n}b^{\ast}_{n})\\
    &=\displaystyle\sum_{n}\dfrac{q^{a+n}+q^{-a-n}}{q^{a}+q^{-a}}\Phi(b_{n}b^{\ast}_{n}).
\end{align*}
Thus $\dfrac{q^{a+n}+q^{-a-n}}{q^{a}+q^{-a}}\Phi(b_{n}b^{\ast}_{n})=0$ for each $n$. Since the Haar state $\Phi$ is faithful, it follows that $b_n=0$ for all $n$. This means that $x=0$. Therefore $\varphi_t$ is faithful.
\end{proof}
\begin{prop}\label{fflrep}
The regular representation $\pi_{\text{reg}}=(\pi_{\text{reg}},H_{\text{reg}})$ is faithful on $C_c(SL_q(2,\mathbb{R})_t)$.
\end{prop}

\begin{proof}
Let $x\in\mathcal{O}_q(S_t^2)\bowtie\mathcal{I}_c$ satisfy $\pi_{\text{reg}}(x)=0$. Write $x=\sum_{n}b_n\Phi_{n}$, where the $b_n$'s are elements in $\mathcal{O}_q(S_t^2)$. Then we have
\begin{align*}
    \pi_{\text{reg}}(x)(\Lambda_{\mathcal{O}_q(S_t^2)}(b)\otimes\Lambda_{\mathcal{I}_c}(y))=0
\end{align*}
for all $b\in\mathcal{O}_q(S_t^2)$ and $y\in\mathcal{I}_c$. In particular letting $b=1$ and $y=\Phi_m$ we have
\begin{align*}
    0=&\pi_{\text{reg}}(\sum_{n}b_n\Phi_{n})(\Lambda_{\mathcal{O}_q(S_t^2)}(1)\otimes\Lambda_{\mathcal{I}_c}(\Phi_m))\\
    =&\sum_{n}\pi_{\text{reg}}(b_n)(\Lambda_{\mathcal{O}_q(S_t^2)}(1)\otimes\Lambda_{\mathcal{I}_c}(\Phi_n\Phi_m))\\
    =&\Lambda_{\mathcal{O}_q(S_t^2)}(b_m)\otimes\Lambda_{\mathcal{I}_c}(\Phi_m)
\end{align*}
for all $m$. Thus $\varphi_t((b_m\Phi_m)^{\ast}b_m\Phi_m)=0$ for all $m$. Since the functional $\varphi_t$ is faithful, $b_m\Phi_m=0$ for each $m$. Hence $x=0$. This proves the proposition.
\end{proof}
Let us consider the category of nondegenerate $\ast$-representations of the convolution algebra $C_c(SL_q(2,\mathbb{R})_t)$. We can find the unit object in this category.
\begin{prop}
The trivial representation is the unit object in the representation category of $C_c(SL_q(2,\mathbb{R})_t)$.
\end{prop}
\begin{proof}
 Supoose $\mathcal{H}$ be a nondegenerate $\ast$-representation of $C_c(SL_q(2,\mathbb{R})_t)$. Note that $\mathcal{O}_q(S_t^2)^{\perp}\coloneqq\{x\in\mathcal{U}_{\infty}\mid x\LHD b=\epsilon(b)x \hspace{1mm}\text{for all}\hspace{1mm}b\in\mathcal{O}_q(S_t^2)\}=\mathcal{I}_{\infty}$. We have
  \begin{align*}
\mathcal{U}_{\infty}\compconj{\Box}\mathds{1}&=\{z\in\mathcal{U}_{\infty}\compconj{\otimes}\mathds{1}\mid(1\otimes\epsilon(b))z(\LHD b\otimes\id)\hspace{1mm}\text{for all}\hspace{1mm}b\in\mathcal{O}_q(S_t^2)\}\\
&=\{z\in\mathcal{U}_{\infty}\mid\epsilon(b)z=z\LHD b \hspace{1mm}\text{for all}\hspace{1mm}b\in\mathcal{O}_q(S_t^2)\}\\
&=\mathcal{O}_q(S_t^2)^{\perp}\\
&=\mathcal{I}_{\infty}.
  \end{align*}
  Thus we have
  \[
\mathcal{H}\boxtimes\mathds{1}=(\mathcal{U}_{\infty}\compconj{\Box}\mathds{1})\compconj{\tens_{\mathcal{I}_{\infty}}}\mathcal{H}\simeq\mathcal{H}.
\]
We also have 
\[
\mathds{1}\boxtimes\mathcal{H}\simeq((\mathcal{U}_{\infty}\compconj{\Box}\mathcal{H})\compconj{\tens_{\mathcal{I}_{\infty}}}(\mathcal{U}_{\infty}\compconj{\Box}\mathds{1}))\Phi_{\mathcal{C}}=(\mathcal{U}_{\infty}\compconj{\Box}\mathcal{H})\Phi_{\mathcal{C}}=\mathcal{H}.
\]
\end{proof}
We give proofs of the following two propositions which are stated in \cite{MT}.
\begin{prop}
 There exists an isomorphism  $\mathcal{M}_{q+q^{-1},a} \simeq (\mathcal{D}_{2}^{+} \oplus \mathcal{D}_{2}^{-}) \oplus \mathds{1}$ of $\mathcal{I}_c$-representations.
\end{prop}
\begin{proof}
    Notice that $\mathcal{M}_{q+q^{-1},a}$ is isomorphic to $(\mathcal{D}_{2}^{+} \oplus \mathcal{D}_{2}^{-}) \oplus \mathds{1}$ as vector spaces. Let $\Psi\colon(\mathcal{D}_{2}^{+} \oplus \mathcal{D}_{2}^{-}) \oplus \mathds{1}\to\mathcal{M}_{q+q^{-1},a}$ be the linear isomorphism $\Psi(\sum_{n\neq0}b_ne_{a+2n}+b_{0}e_a)=\sum_{n\neq0}b_ne_{a+2n}+b_{0}e_a$. Let $x=\sum_{n}a_{n}\Phi_{n}\in\mathcal{I}_c$ be given. Let $\xi=\sum_{n\neq0}b_ne_{a+2n}+b_{0}e_a\in(\mathcal{D}_{2}^{+} \oplus \mathcal{D}_{2}^{-}) \oplus \mathds{1}$ be given. Then we have the following commutative diagram.
  \[
\xymatrix{
        (\mathcal{D}_{2}^{+} \oplus \mathcal{D}_{2}^{-}) \oplus\mathds{1} \ar[r]^{ \Psi} \ar[d]^{\pi(x)}
    & \mathcal{M}_{q+q^{-1},a} \ar[d]^{\pi(x)} \\
         (\mathcal{D}_{2}^{+} \oplus \mathcal{D}_{2}^{-}) \oplus\mathds{1} \ar[r]^{\Psi}
    & \mathcal{M}_{q+q^{-1},a}
   }
\xymatrix{
  \xi  \ar@{|->}[r] \ar@{|->}[d]  &  \sum_{n\neq0}b_ne_{a+2n}+b_{0}e_a\ar@{|->}[d] \\
  \xi \ar@{|->}[r]    & \sum_{n\neq0}a_nb_ne_{a+2n}+a_0b_{0}e_a
}
\]   where $\pi$ denotes the representation $\pi_{(\mathcal{D}_{2}^{+}\oplus \mathcal{D}_{2}^{-}) \oplus\mathds{1}}$. Hence $\Psi$ is a homomorphism between the representations of $\mathcal{I}_c$. 
    
\end{proof}
However, the sequence does not split in the category of $ \mathcal{O}_q(S^{2}_t) \bowtie \mathcal{I}_c$:
\begin{prop}\label{nslem}
There is no splitting $\lambda \colon \mathds{1} \to \mathcal{M}_{q+q^{-1}, a} $ in the representation category $\Rep(\mathcal{O}_q(S^{2}_t) \bowtie \mathcal{I}_c)$.
\end{prop}
\begin{proof}
Suppose the contrary. The splitting $\lambda$ is uniquely determined by the vector $\lambda(1)$. Applying the map $\Phi_{C}$, we have $\lambda(1)=ce_{a}$ for some $c \in \mathbb{C} \setminus{\{0\}}$ since the subspace of $\mathcal{M}_{q+q^{-1}, a}$ consisting of the vectors with weight $[a]$ is one-dimensional.
   \[
\xymatrix{
        \mathds{1} \ar[r]^{\lambda} \ar[d]^{\Phi_{C}}
    & \mathcal{M}_{q+q^{-1},a} \ar[d]^{\Phi_{C}} \\
         \mathds{1} \ar[r]^{\lambda}
    & \mathcal{M}_{q+q^{-1},a}
   }
\xymatrix{
  1  \ar@{|->}[r] \ar@{|->}[d]  &  \lambda(1)\ar@{|->}[d] \\
  1 \ar@{|->}[r]    & \lambda(1)
}
\]
\quad Next, we apply the operator $T_{a}^{+}$, which sends $e_{a}$ to $e_{a+2}$. Then we must have $\epsilon(T_{a}^{+})e_{a}=e_{a+2}$. This is a contradiction.

 \[
\xymatrix{
        \mathds{1} \ar[r]^{\lambda} \ar[d]^{T_{a}^{+}}
    & \mathcal{M}_{q+q^{-1},a} \ar[d]^{T_{a}^{+}} \\
         \mathds{1} \ar[r]^{\lambda}
    & \mathcal{M}_{q+q^{-1},a}
   }
\xymatrix{
  1  \ar@{|->}[r] \ar@{|->}[d]  &  ce_{a}\ar@{|->}[d] \\
   \epsilon(T_{a}^{+})1 \ar@{|->}[r]    &  c\epsilon(T_{a}^{+})e_{a}=ce_{a+2}
}\]

\end{proof}
\section{Generating functionals and $1$-cocycles}
In this section we construct a generating functional on a Podle\'{s} sphere which is proper. Throughout this section let $\mathcal{A}$ be a CQG Hopf $*$-algebra and $\mathcal{B}$ be a unital left coideal $*$-subalgebra of $\mathcal{A}$. We take the convention in \cite{MT}.
\begin{defn}[generating functionals]
A linear map $L \colon \mathcal{B}\to\mathbb{C}$ is called a \emph{generating functional} if $L$ satisfies the following conditions:\\
(1)$L(a^{*}a)\geq0$ if $\epsilon(a)=0$ (\emph{conditional positivity})\\
(2)$L(a^{*})=\compconj{L(a)}$ for all\quad$a\in\mathcal{B}$ (\emph{Hermitian})\\
(3)$L(1)=0$ (\emph{vanising at $1$})
\end{defn}
By the formula $\langle a, b \rangle_{_{_{L}}}\coloneqq L((a-\epsilon(a)1)^{*}(b-\epsilon(b)1))\quad(a,b\in\mathcal{B})$, we define a positive sesquilinear form on $\mathcal{B}$. Put, just as in the case of the usual GNS construction for states on $C^{*}$-algebras, put $N_{L}\coloneqq\{a\in\mathcal{B} \mid  \langle a, a \rangle_{_{_{L}}}=0\}$ and let $\mathcal{H}_{L}\coloneqq\mathcal{B}/N_{L}$. Then we get a pre-Hilbert space. We denote the canonical map $\mathcal{B}\twoheadrightarrow  \mathcal{H}_{L}$ by $C_{L}$
\begin{prop}[GNS type construction \cite{Mic}]
\label{GNS}
Let $\mathcal{B}$ and $L$ be as above. Then we can define a map $\pi_{L} \colon \mathcal{B} \to \End(\mathcal{H}_{L})$ by $\pi_{L}(a)C_{L}(b)=C_{L}(ab)-C_{L}(a)\epsilon(b)$, where $a,b\in\mathcal{B}$. In addition, $\pi_{L}$ is a $*$-homomorphism and $C_{L}$ is a $(\pi_{L}, \epsilon)$-1-cocycle.
\end{prop}
Let $0<\lambda<q+q^{-1}$. Let $\mathcal{L}_{\lambda}^+=(\pi_{\lambda},\mathcal{L}_{\lambda}^+)$. 
Let $e_{a}^{\lambda}$ be the unit vector of weight $[a]$ in the space $\mathcal{L}_{\lambda}^+$. Put $\omega_{\lambda}(\bullet)\coloneqq \langle e_{a}^{\lambda}, \pi_{\lambda}(\bullet) e_{a}^{\lambda}\rangle$, where $\langle \bullet, \bullet \rangle$ is the usual inner product. Let $\omega_{\lambda}(\bullet)\coloneqq \langle e_{a}^{\lambda}, \pi_{\lambda}(\bullet) e_{a}^{\lambda}\rangle$.
We compute $\omega_{\lambda}(u^{n}_{[a+2i], [a]})$. First we prove the following lemma.
\begin{lem}\label{wt} We have 
\begin{align*}
\langle e_{a}^{\lambda}, \pi_{\lambda}(u^{n}_{[a+2i], [a]})e_{a}^{\lambda}\rangle=
 \begin{cases}
    \langle e_{a}^{\lambda}, \pi_{\lambda}(u^{n}_{[a],[a]})e_{a}^{\lambda}\rangle & \text{if $ i=0 $,} \\
    0  & \text{if $i\neq0$.} \\
  \end{cases}
\end{align*}
for all $0<\lambda<q+q^{-1}$.
\end{lem}
\begin{proof}
By the commuting relation, we have \begin{align*}
&\langle e_{a}^{\lambda}, \pi_{\lambda}(u^{n}_{[a+2i], [a]})e_{a}^{\lambda}\rangle\\
&=\langle e_{a}^{\lambda}, \pi_{\lambda}(u^{n}_{[a+2i], [a]}\Phi_{C})e_{a}^{\lambda}\rangle\\
&=\langle e_{a}^{\lambda}, \pi_{\lambda}((\Phi_{C(1)})(u^{n}_{[a+2i], [a]}\lhd\Phi_{C(2)}))e_{a}^{\lambda}\rangle\\
&=\langle \pi_{\lambda}(\Phi_{C(1)})e_{a}^{\lambda}, \pi_{\lambda}(u^{n}_{[a+2i], [a]}\lhd\Phi_{C(2)})e_{a}^{\lambda}\rangle\\
&=\langle \pi_{\lambda}(\Phi_{C})e_{a}^{\lambda}, \pi_{\lambda}(u^{n}_{[a+2i], [a]}\lhd\Phi_{C})e_{a}^{\lambda}\rangle\\
&=\langle e_{a}^{\lambda}, \pi_{\lambda}(u^{n}_{\Phi_{\mathcal{C}}\eta_{[a+2i]}^{n}, \eta_{[a]}^{n}})e_{a}^{\lambda}\rangle\\
&=
\begin{cases}
    \langle e_{a}^{\lambda}, \pi_{\lambda}(u^{n}_{[a],[a]})e_{a}^{\lambda}\rangle & \text{if $ i=0 $,} \\
    0  & \text{if $i\neq0$.} \\
  \end{cases}
\end{align*}
\end{proof}
By Lemma \ref{wt}, there exists some $a_{0}^{\lambda, n} \in \mathbb{C}$ such that $\omega_{\lambda}(u^{n}_{[a+2i], [a]})=\delta_{0,i}a_{0}^{\lambda, n}$, where $\delta_{i,j}$ is a Kronecker delta. This $a_{0}^{\lambda, n}$ can be expressed in terms of Askey--Wilson polynomials. Namely, by Lemma \ref{AW}, we can write $u^{n}_{[a], [a]}=P_{n}(u^{1}_{[a],[a]})$, where $P_{n}$ is a certain normalization of the Askey--Wilson polynomial. Then we have \begin{align*}
    a_{0}^{\lambda, n}&=\langle e_{a}^{\lambda}, \pi_{\lambda}(P_{n}(u^{1}_{[a], [a]}))e_{a}^{\lambda} \rangle\\
    &=\langle e_{a}^{\lambda}, \pi_{\lambda}(P_{n}(\tfrac{A_{a}+(q+q^{-1})^{-1}\llbracket a\rrbracket^{2}}{q+q^{-1}+(q+q^{-1})^{-1}\llbracket a\rrbracket^2})e_{a}^{\lambda} \rangle\\
    &=\langle e_{a}^{\lambda}, \pi_{\lambda}(P_{n}(\tfrac{\lambda+(q+q^{-1})^{-1}\llbracket a\rrbracket^{2}}{q+q^{-1}+(q+q^{-1})^{-1}\llbracket a\rrbracket^2}))e_{a}^{\lambda} \rangle\\
&=P_{n}(\tfrac{\lambda+(q+q^{-1})^{-1}\llbracket a\rrbracket^{2}}{q+q^{-1}+(q+q^{-1})^{-1}\llbracket a\rrbracket^2}).
\end{align*}
From this and simple computation, we have 
\begin{thm}\label{gen}
For each $b\in \mathcal{O}_{q}(S_{t}^{2})$, the limit $\displaystyle\lim_{\lambda \to q+q^{-1}}  \dfrac{\omega_{\lambda}(b)-\epsilon(b)}{q+q^{-1}-\lambda}$ converges and defines a generating functional on $\mathcal{O}_{q}(S_{t}^{2})$. Furthermore, the limit is expressed as a constant multiple of the derivation of some normalization of Askey--Wilson polynomial (cf.\ Lemma \ref{AW}). 
\end{thm}
\begin{proof}
Since the set $\{u^{n}_{[a+2i], [a]}\mid n\in\mathbb{Z}_{\geq 0}, i=-n,\ldots,n\}$ forms a basis for $\mathcal{O}_{q}(S_{t}^{2})$, it suffices to show that  $\displaystyle\lim_{\lambda \to q+q^{-1}}  \dfrac{\omega_{\lambda}(u^{n}_{[a+2i], [a]})-\epsilon(u^{n}_{[a+2i], [a]})}{q+q^{-1}-\lambda}$ converges and defines a generating functional for each $n\in\mathbb{Z}_{\geq 0}$ and $i=-n,\ldots,n$.\par
 It is clear that this limits defines a generating functional if the limits exists. Let us show that the limit exists. By Lemma \ref{wt}, this limits is equal to $0$ if $i\neq0$. If $i=0$, then by Lemma \ref{wt} again 
\begin{align*}
&\displaystyle\lim_{\lambda \to q+q^{-1}}  \dfrac{\omega_{\lambda}(u^{n}_{[a+2i], [a]})-\epsilon(u^{n}_{[a+2i], [a]})}{q+q^{-1}-\lambda}\\
&=\displaystyle\lim_{\lambda \to q+q^{-1}}  \dfrac{\omega_{\lambda}(u^{n}_{[a], [a]})-\epsilon(u^{n}_{[a], [a]})}{q+q^{-1}-\lambda}\\
&=\displaystyle\lim_{\lambda \to q+q^{-1}}  \dfrac{P_{n}(\tfrac{\lambda+(q+q^{-1})^{-1}\llbracket a\rrbracket^{2}}{q+q^{-1}+(q+q^{-1})^{-1}\llbracket a\rrbracket^2})-1}{q+q^{-1}-\lambda}\\
&=-P'_{n}(1).
\end{align*}
Thus the limit exists for each $b\in\mathcal{O}_{q}(S_t^2)$.
\end{proof}
\begin{Rem}
These calculations and \cite{MT} show that our generating functional $L$ is proper. 
\end{Rem}
We put $L \coloneqq \{a+2\}\{a\}^{-1}(\{2a+1\}+q+q^{-1})^{-1}\displaystyle\lim_{\lambda \to q+q^{-1}} \dfrac{\omega_{\lambda}(\bullet)-\epsilon(\bullet)}{q+q^{-1}-\lambda}$, then, by Proposition \ref{GNS}, we can construct a $1$-cocycle $C_L$ on $\mathcal{O}_{q}(S_{t}^{2})$. We show $1$-cocycles $C$ and $C_L$ are unitarily equivalent.
\begin{prop}
We equip $\mathcal{D}_{2}^{+} \oplus \mathcal{D}_{2}^{-}$ with the invariant inner product constructed in \cite[Subsection 4.2]{MT}. Then There exists a unitary operator $U\colon \mathcal{H}_{L} \to \mathcal{D}_{2}^{+} \oplus \mathcal{D}_{2}^{-}$ satisfying $C(\bullet)=UC_{L}(\bullet)$.
\end{prop}
\begin{proof}
Let $U\colon\mathcal{H}_{L} \to \mathcal{D}_{2}^{+}\oplus\mathcal{D}_{2}^{-}$ be the correspondence $C_{L}(x)\mapsto C(x)$. We show that $U$ is well-defined and unitary. Since the surjectivity of $U$ is clear, it suffices to show that $\langle C(x), C(x)\rangle_{_{\mathcal{M}_{q+q^{-1}, a}}}=\langle C_{L}(x), C_{L}(x)\rangle_{L}$  for all $x=\sum_{n\in\mathbb{Z}_{\geq0}}\sum_{i=-n,\ldots, n}b^{n}_{i}u_{[a+2i],[a]}$, where $b^{n}_{i}\in\mathbb{C}$ and $\langle \bullet, \bullet\rangle_{_{\mathcal{M}_{q+q^{-1}, a}}}$ is the invariant inner product constructed in \cite[Subsection 4.2]{MT}. Notice that $C_{L}(u^n_{[a], [a]})=C(u^n_{[a], [a]})=0$. Then we have 

\begin{align*}
   & \langle C(x), C(x)\rangle_{_{\mathcal{M}_{q+q^{-1}, a}}}\\
   &=\sum_{n,m\in\mathbb{Z}_{\geq0}}\sum_{i,j=-n,\ldots,n}\compconj{a_i^n}a_j^m\langle C(u^n_{[a+2i], [a]}), C(u^m_{[a+2j], [a]})\rangle_{_{\mathcal{M}_{q+q^{-1}, a}}}\\
   &=\sum_{n,m\in\mathbb{Z}_{>0}}\sum_{i,j=-n,\ldots,n}\compconj{a_i^n}a_j^m\langle C(u^n_{[a+2i], [a]}), C(u^m_{[a+2j], [a]})\rangle_{_{\mathcal{M}_{q+q^{-1}, a}}}\\
   &=\sum_{n,m\in\mathbb{Z}_{>0}}\sum_{i,j=-n,\ldots,n}\lim_{\lambda}\compconj{a_i^n}a_j^m\tfrac{\{a+2\}}{\{a\}(\{2a+1\}+q+q^{-1})}\\
   &\cdot\dfrac{\langle \pi_{\lambda}(u^n_{[a+2i], [a]})e_{a}-\epsilon(u^n_{[a+2i], [a]})e_{a},  \pi_{\lambda}(u^n_{[a+2i], [a]})e_{a}-\epsilon(u^n_{[a+2i], [a]})e_{a}\rangle_{_{_{\lambda}}}}{q+q^{-1}-\lambda}
\end{align*}
\begin{align*}
   &=\sum_{n,m\in\mathbb{Z}_{>0}}\sum_{i,j=-n,\ldots,n}\lim_{\lambda}\compconj{a_i^n}a_j^m\tfrac{\{a+2\}}{\{a\}(\{2a+1\}+q+q^{-1})}\tfrac{\langle \pi_{\lambda}(u^n_{[a+2i], [a]})e_{a},  \pi_{\lambda}(u^n_{[a+2i], [a]})e_{a}\rangle_{_{_{\lambda}}}}{q+q^{-1}-\lambda}\\
 &=\sum_{n,m\in\mathbb{Z}_{>0}}\sum_{i,j=-n,\ldots,n}\compconj{a_i^n}a_j^m L(u^{n\ast}_{[a+2i], [a]}u^m_{[a+2j], [a]})\\
   &=\sum_{n,m\in\mathbb{Z}_{>0}}\sum_{i,j=-n,\ldots,n}\compconj{a_i^n}a_j^m\langle C_{L}(u^n_{[a+2i], [a]}), C_{L}(u^m_{[a+2j], [a]})\rangle_{_{L}}\\
   &=\sum_{n,m\in\mathbb{Z}_{\geq0}}\sum_{i,j=-n,\ldots,n}\compconj{a_i^n}a_j^m\langle C_{L}(u^n_{[a+2i], [a]}), C_{L}(u^m_{[a+2j], [a]})\rangle_{_{L}}\\
   &=\langle C_{L}(x), C_{L}(x)\rangle_{_{L}}
\end{align*}
This completes the proof.
\end{proof}
We introduce the notion of the Gaussian part and the non-Gaussian part of a given $1$-cocycle.
\begin{defn}\label{Gauss}\cite[Chapter 5]{Mic}\cite[Definition 2.3.]{FGT}
Let $\mathcal{B}$ be a unital left coideal $*$-subalgebra of a $*$-bialgebra. Let $\pi \colon \mathcal{B} \to \End(\mathcal{H})$ be a $*$-representation of $\mathcal{B}$ on a pre-Hilbert space $\mathcal{H}$. A $(\pi, \epsilon)$-$1$-cocycle $C \colon \mathcal{B} \to \mathcal{H}$ on $\mathcal{B}$ is called \emph{Gaussian} if  $C(bc)=\epsilon(b)C(c)+C(b)\epsilon(c)$ for all $b,c\in \mathcal{B}$.
\end{defn}
Let $\mathcal{B}$ and $\pi$ be as above. Let $C$ be a surjective $(\pi, \epsilon)$-$1$-cocycle. Let $H$ denotes the completion of $\mathcal{H}$. Let $H_{G}=\{(\pi(b)-\epsilon(b))C(c) \mid b,c\in \mathcal{B}\}^{\perp}$ and $H_{NG}=\overline{\{(\pi(b)-\epsilon(b))C(c) \mid b,c\in \mathcal{B}\}}$. Let $P_{G}$ and $P_{NG}$ be the orthogonal projections onto $H_{G}$ and $H_{NG}$ respectively. We put $\mathcal{H}_{G}\coloneqq P_{G}\mathcal{H}$ and $\mathcal{H}_{NG}\coloneqq P_{NG}\mathcal{H}$. Define representations $\pi_{G}$ and $\pi_{NG}$ on $\mathcal{H}_{G}$ and $\mathcal{H}_{NG}$ respectively by the following formula:
 \begin{align*}
        \pi_{G}(b)P_{G}C(c)&=\epsilon(b)P_{G}C(c)\quad(b,c \in\mathcal{B})\\
       \pi_{NG}P_{NG}C(c)&=\pi(b)C(c)-\epsilon(b)P_{G}C(c)\quad(b,c \in\mathcal{B})
    \end{align*}
Then $C_{G}\coloneqq P_{G}\circ C$ is a $(\pi_{G}, \epsilon)$-$1$-cocycle (in particular it is Gaussian), and  $C_{NG}\coloneqq P_{NG}\circ C$ is a $(\pi_{NG}, \epsilon)$-$1$-cocycle. Furthermore we have a decomposition $C=C_{G}+C_{NG}$. We call $C_{G}$ and $C_{NG}$ \emph{the Gaussian part of $C$} and \emph{the non-Gaussian part of C} respectively. 
\begin{defn}\label{NG}
We say $C$ is \emph{purely non-Gaussian} if $\mathcal{H}_{G}=0$.    
\end{defn}
\begin{prop}
Our $1$-cocycle $C$ is purely non-Gaussian.
\end{prop}
\begin{proof}
Since $\{(\pi(b)-\epsilon(b))C(c) \mid b,c\in \mathcal{O}_q(S_t^2)\}$  contains all of $e_{a+2n}$ for all non-zero integers $n$, we have $\mathcal{H}_{G}=0$.
\end{proof}
\section{Some more observations}
In this section we propose some observations of the property the quantum $SL(2,\mathbb{R})$ has.\par
A topological group $G$ is \emph{liminal} if the full group $C^{\ast}$-algebra $C^{\ast}(G)$ is liminal. Hence in order to prove that $G$ is of type $1$, it suffices to show that for every irreducible unitary representation $(\pi,H)$ of $C^{\ast}(G)$ we have $\pi(C^{\ast}(G))\subset\mathbb{K}(H)$. In this subsection we show that the object which should be regarded as $C^{\ast}_u(SL_q(2,\mathbb{R})_t)$ is `liminal' by showing that for every irreducible unitary representation $(\pi,H)$ of $C^{\ast}_u(SL_q(2,\mathbb{R})_t)$, we have $\pi(C^{\ast}_u(SL_q(2,\mathbb{R})_t)) \subset\mathbb{K}(H)$. At first we show that the full `group' $C^{\ast}$-algebra $C^{\ast}_u(SL_q(2,\mathbb{R})_t)$ exists.
We define the full `group' $C^{\ast}$-algebra $C^{\ast}_u(SL_q(2,\mathbb{R})_t)$ by the completion of the algebra $C_c(SL_q(2,\mathbb{R})_t)\coloneqq\mathcal{O}_q(S_t^2)\bowtie\mathcal{I}_c$ (of `continuous functions' on $SL_q(2,\mathbb{R})$ of compact support) with respect to the $C^{\ast}$-norm $||\bullet||_{u}$ defined by
\begin{align*}
    ||f||_{u}\coloneqq\sup\{||\pi(f)||;\pi\hspace{1mm}\text{is an irreducible} \ast\text{-representation of}\hspace{1mm}C_c(SL_q(2,\mathbb{R})_t)\}.
\end{align*}
\begin{prop}$||\bullet||_{u}$ is a well-defined $C^{\ast}$-norm. Hence the $C^{\ast}$-algebra $C^{\ast}_u(SL_q(2,\mathbb{R})_t)$ exists.
\end{prop}
\begin{proof}
This proposition follows from the fact that the regular representation is faithful on $C_c(SL_q(2,\mathbb{R})_t)$(Proposition \ref{fflrep}).
\end{proof}
\begin{thm}
The quantum $SL_q(2,\mathbb{R})$ is liminal, meaning that its full `group' $C^{\ast}$-algebra is liminal.
\end{thm}
\begin{proof}
Let $(\pi,H)$ be a unitary irreducible $\ast$-representation of $C^{\ast}_u(SL_q(2,\mathbb{R}))$ on a Hilbert space $H$. Let $\mathcal{H}\coloneqq\pi(\mathcal{I}_c)H$. Note that $\mathcal{H}$ is a preHilbert space whose completion is $H$ and $\pi$ induces an irreducible $\ast$-representation of $C_c(SL_q(2,\mathbb{R}))$. Let $f\in C_c(SL_q(2,\mathbb{R}))$ be given. By \cite[Theorem 2.14]{MT}, we can take an orthonormal basis $\{e_n\}_{n\in S}$ of $\mathcal{H}$ such that $\sqrt{-1}\pi(B_t)e_n=[a+2n]e_n$ for each $n$, where $S\coloneqq\mathbb\{n\in\mathbb{Z}\mid \Phi_n\mathcal{H}\neq0\}$. Write $f=\sum_{n\in F}b_n\Phi_n$, where $F\subset S$ is a finite subset and $b_n\in\mathcal{O}_q(S_t^2)$ for each $n$. Notice that $\pi(\Phi_n)$ is rank $1$ projection. Then $\pi(f)$ is a sum of finite rank operator and hence $\pi(f)$ is compact.\par
Let $\pi$ be an irreducible $\ast$-representation of the $C^{\ast}$-algebra $C^{\ast}_u(SL_q(2,\mathbb{R}))$ on a Hilbert space $H$. Let $f\in C^{\ast}_u(SL_q(2,\mathbb{R}))$ and $\epsilon>0$ be given, then there exists some $g\in C_c(SL_q(2,\mathbb{R}))$ such that $||f-g||_{\text{full}}<\epsilon$. Then
\begin{align*}
    ||\pi(f)-\pi(g)||\leq||f-g||_{\text{full}}<\epsilon.
\end{align*}
Hence the operator $\pi(f)$ is compact. This proves the theorem.
\end{proof}
\begin{cor}
The von Neumann algebra $C_c(SL_q(2,\mathbb{R}))''$ is not a factor.
\end{cor}

\section{Research directions}
There are many interesting problems on the quantum $SL(2,\mathbb{R})$. The following are some of them. 
\begin{itemize}
\item To compute the growth of each of $1$-cocycles $C_{\pm}$: In our previous work \cite{MT} we decomposed $1$-cocycle $C$ into the sum of $1$-cocycles $C_+$ and $C_-$. We know the growth of $C$, however we do not yet know the growth of each of $C_{\pm}$. How should we compute these?
\item To give definitions of group theoretical properties of the quantum $SL(2,\mathbb{R})$ such as unimodularity: Recall that a Hopf algebra $H$ is called \emph{unimodular} if the following two sets, each of which is called \emph{the space of left integrals} and \emph{the space of right integrals}, respectively, coincide:
\begin{align*}
    I_l&\coloneqq\{x\in H\mid ax=\epsilon(a)x\hspace{1mm}\text{for all}\hspace{1mm}a\in H\},\\
    I_r&\coloneqq\{x\in H\mid xa=\epsilon(a)x\hspace{1mm}\text{for all}\hspace{1mm}a\in H\}.
\end{align*}
In this sense $U_q(\mathfrak{su}(2))$ is unimodular.
Let 
 \begin{align*}
     I_l&\coloneqq\{x\in U_q(\mathfrak{sl}(2,\mathbb{R})_t)\mid ax=\epsilon(a)x\hspace{1mm}\text{for all}\hspace{1mm}a\in U_q(\mathfrak{sl}(2,\mathbb{R})_t)\},\\
    I_r&\coloneqq\{x\in U_q(\mathfrak{sl}(2,\mathbb{R})_t)\mid xa=\epsilon(a)x\hspace{1mm}\text{for all}\hspace{1mm}a\in U_q(\mathfrak{sl}(2,\mathbb{R})_t)\}.
 \end{align*}
These spaces coincide:
\begin{prop}
The space $I_l$ is zero. Therefore the space $I_r$ is zero.
\end{prop}
\begin{proof}
Note that $I\subset\mathcal{I}$. Let $x\in I_l$. Write $x=\sum_{n}\Phi_{n}b_n$, where $b_n\in\mathcal{O}_q(S_t^2)$. Then by the commuting relations  $\Phi_{0}x=\Phi_{0}b_0=\sum_{n}b'_n\Phi_n$ for some $b'_n\in\mathcal{O}_q(S_t^2)$. Let $y\in\mathcal{O}_q(S_t^2)$. Then we have
\begin{align*}
    \displaystyle\sum_{n}yb'_{n}\Phi_n&=y\Phi_{0}x=\epsilon(y)\Phi_{0}x=\displaystyle\sum_{n}\epsilon(y)b'_{n}\Phi_n.
\end{align*}
Hence $yb'_n=\epsilon(y)b'_n$ for each $n$. Since the Podle\'{s} sphere $\mathcal{O}_q(S_t^2)$ is an integral domain, we have $b'_n=0$ for all $n$. Thus $\Phi_{0}x=0$. In a similar way, we can show that $\Phi_{m}x=0$ for all $m\in\mathbb{Z}$. Hence we have $I_l=0$. From this, it follows that $I^{\ast}_r\subset I_l=0$.  
\end{proof}
Is this an appropriate way of defining the unimodularity for the quantum $SL(2,\mathbb{R})$? 
    \item To give definitions and characterizations of properties such as Haagerup property, property (T), amenability et cetera:
We know that
\begin{itemize}
    \item The trivial representation is not isolated. 
    \item There exists a $1$-cocycle which is not a coboundary.
\end{itemize}
From these facts, we should consider that the quantum $SL(2,\mathbb{R})$ is not of property (T). We have a proper $1$-cocycle on the quantum $SL(2,\mathbb{R})$ \cite{MT}. Thus we should consider that the quantum $SL(2,\mathbb{R})$ is of Haagerup property. Considering this situations, probably we may define the property (T) and Haagerup property for quantized locally compact real Lie groups as coideals through the languages of cocycles. How should we define other properties? If we can define these concepts, how can we prove these properties? In \cite{And} the concepts of amenability and coamenability on coideals is considered. Can we show that the quantum $SL(2,\mathbb{R})$ is not (co)amenable in some sense (for instance in the sense of \cite{And})? 
    \item Schoenberg correspondence and to give a probabilistic interpretation of our $1$-cocycle: On the bialgebras $1$-cocycles, generating functionals, L\'{e}vy process et cetera are in one to one correspondence because the convolution product is a well-defined notion. As in \cite{DFW}, if coideals in question are expected, then the convolution product is well-defined. Thus, in the case of the quantum $SL(2,\mathbb{R})$, perhaps we can consider and prove this correspondence if we take the standard Podle\'{s} sphere. However how should we struggle this problem when we take the non-standard Podle\'{s} sphere? Also is there a probability theoretical interpretation of our purely non-Gaussian $1$-cocycle?
\end{itemize}
 

\end{document}